\documentclass[12pt]{amsart}

\usepackage{amssymb}
\newtheorem{defin}{Definition}
\newtheorem{propo}{ Proposition}
\newtheorem{lemme}{Lemma}
\newtheorem{coro}{Corollary}
\newtheorem{theorem}{Theorem}

\begin{document}

\title [ strictly positive definite and logarithmically completely monotonic functions]{A function  Class of strictly positive definite and \\
logarithmically completely monotonic functions related to the modified Bessel functions}%

\author[ J. El Kamel, K. Mehrez]{ Jamel El Kamel \qquad and \qquad Khaled Mehrez }
 \address{Jamel El Kamel. D\'epartement de Math\'ematiques fsm. Monastir 5000, Tunisia.}
 \email{jamel.elkamel@fsm.rnu.tn}
 \address{Khaled Mehrez. D\'epartement de Math\'ematiques IPEIM. Monastir 5000, Tunisia.}
 \email{k.mehrez@yahoo.fr}
\begin{abstract}
In this paper we give some conditions for a class of functions related to Bessel functions 
to be positive definite  or strictly positive definite . We present some properties and 
relationships involving logarithmically completely monotonic functions and strictly 
positive definite  functions. In particular, we are
interested with the modified Bessel functions.
 
  \end{abstract}
\maketitle
{\it keywords:} Bessel functions, Positive definite functions, Completely monotonic functions, 
logarithmically Completely monotonic functions.\\

\section{Introduction}
A complex valued continuous function $f$ is said positive definite (resp. strictly positive definite ) 
 on $\mathbb{R}$ if for every  reals numbers $x_1, x_2, ...x_n$ and every complex numbers 
 $z_1, z_2, ...z_n$ not all zero, the inequality 
 $$ \displaystyle \sum_{j=1}^{N}\sum_{k=1}^{N} z_j \bar{z_k} f(x_j-x_k) \geq 0 \quad{\rm(resp. >0)}$$
 
\noindent holds true ( see \cite{Fa} ).\\
We denote by $\mathcal{P}$ (resp. $\mathcal{P}^s $) the class of such functions.\\
 \noindent Bochner's theorem \cite{Bo} characterizes positive definite functions as Fourier transform of 
 nonnegative finite Borel measure on the real line.\\
 
 \noindent A function $f$ is said to be completely monotonic (CM) on an interval $I\subset\mathbb{R}$, 
 if $f\in C(I)$ has derivatives of all orders on $I^0$ (the interior of $I$) and, for all $n\in \mathbb{N}$,\\
 $$\displaystyle (-1)^n f^{(n)}(x)\geq 0, \quad x\in I^0;\quad n\in \mathbb{N}.$$
 \noindent The class of all  completely monotonic functions on  $I$ is denoted by $CM(I)$.\\
 \noindent  A function $f$ is said to be logarithmically completely monotonic (LCM) on an interval $I\subset\mathbb{R}$, if 
 $$f>0, \quad f\in C(I),$$ 
 has derivatives of all orders on $I^0$ and, for all $n\in \mathbb{N}\setminus \{0\}$,\\
 $$\displaystyle (-1)^n [{\rm ln} f(x)]^{(n)}\geq 0, \quad x\in I^0;\quad  n\geq 1.$$
 \noindent The class of all logarithmically completely monotonic functions on  $I$ is denoted by $LCM(I)$. We have $LCM(I)\subset CM(I)$.\\
 
\noindent Berstein's theorem \cite{Be} asserts that $f$ is completely monotonic function if and only if $f$ is the Laplace transform of nonnegative finite Borel measure on $[0,\infty[$.\\
In 1938, Schoenberg \cite{Sc} studied the completely monotonic functions and proved that theses functions are closely related to positive definite functions.\\
H. Wendland \cite{We} was interested by strictly  positive definite functions, and present a complete 
characterization  of radial functions as being  strictly  positive definite on every $\mathbb{R}^d$.\\
L\'evy Kinchin theorem asserts that a probability measure $d\mu$ supported on $[0,\infty[$ is infinitely divisible if and only if it's Laplace transform of an logarithmically completely monotonic function.\\ 

 In this paper, we consider a function class related to the modified Bessel functions. In the first part we prove that among theses functions there are whose positive definite and strictly positive definite. In the second part
 we present some properties and relationships involving logarithmically completely monotonic functions and strictly 
positive definite  functions. In particular, we are
interested with the modified Bessel functions of the second kind.\\
 
Our paper is organized as follows : in Section 2, we present some
preliminaries results and notations that will be useful in the
sequel. In Section 3, we present some properties of the Bessel transform, the Bessel translation operator and the Bessel convolution product. All of these
results can be found in \cite{Fi} and \cite{Bek}.
 In Section 4,  we give some conditions for a class of functions related to Bessel functions 
 to be positive definite or  strictly positive definite. Much attention is devoted to the Bessel function, 
 the modified Bessel function of second kind and the Bessel-Fourier transform. In Section 5,  We present some properties and relationships involving logarithmically completely monotonic functions and strictly 
positive definite  functions. As applications, in Section 6, using Schoenberg theorem \cite{Sc} and Wendland theorem \cite{We}, we prove 
logarithmically monotonicity for a class of functions related to the modified Bessel functions of second kind. In particular, it' is well known that the function 
$\displaystyle \frac{1}{x^{\frac{\alpha}{2}}K_\alpha(\sqrt{x})}$  
in not completely monotonic, we will prove that the  function 
$\displaystyle \frac{1}{x^{\frac{\alpha+1}{2}}K_\alpha(\sqrt{x})}$ 
is even logarithmically completely monotonic. We note that Ismail \cite{Ism} prove that the function 
$\displaystyle \frac{1}{e^{\sqrt{x}}x^{\frac{\alpha}{2}}K_\alpha(\sqrt{x})}$ 
is completely monotonic. We derive some new inequalities for  the modified Bessel functions of second kind.

 \section { Notations and preliminaries}
 \indent 
 
 The normalized Bessel function of index $\alpha>-\frac{1}{2}$ is the even function defined by: 
 \begin{equation}
 \displaystyle j_\alpha (x)=\left\{\begin{array}{cc}
 2^\alpha \Gamma (\alpha+1) \frac{J_\alpha (x)}{x^\alpha},\quad x>0,\\
 1,\qquad\qquad\qquad\quad x=0
 \end{array}\right.
 \end{equation}
 where $J_\alpha$ is the Bessel function of first kind and index $\alpha$. Thus 
\begin{equation}
 \displaystyle j_\alpha (x)=\sum_{n=0}^{+\infty} \frac{(-1)^n \Gamma (\alpha+1) }{n! \Gamma (n+\alpha+1)} (\frac{x}{2})^{2n},
 \end{equation}
 
\begin{equation}
 \displaystyle \left|j_\alpha (x)\right|\leq 1,\quad x\geq 0, \quad\alpha>-\frac{1}{2},
 \end{equation}
 
  \begin{equation}
  \displaystyle j_\alpha (x)=\frac{1}{2^{\alpha-1} \sqrt{\pi}\Gamma (\alpha+\frac{1}{2})} 
  \int_0^1 \left(1-t^2\right)^{\alpha-\frac{1}{2}} \cos (xt) dt.
 \end{equation} 
  \\
\noindent  The modified Bessel function $I_\alpha$ of first kind and index $\alpha$ is defined by:
   \begin{equation}
  \displaystyle I_\alpha (x)=\sum_{n=0}^{+\infty} \frac{(\frac{x}{2})^{2n+\alpha}}{n! \Gamma (n+\alpha+1)},\quad
   \alpha \neq -1,-2,...; x\in \mathbb{R}.
   \end{equation}
 On can see  easily that 
  $$\displaystyle I_\alpha (x)>0,\quad \forall \alpha >-1 ,\quad \forall x>0.$$
  
\noindent  The modified Bessel function $K_\alpha$ of second kind (called sometimes Macdonald function) and index $\alpha$ is defined by:
  \begin{equation}
 \displaystyle K_\alpha (x)= \frac{\pi}{2}\frac{I_{-\alpha} (x)-I_\alpha (x)}{\sin \alpha \pi},
  \end{equation}
 where the right-hand side of this equation is replaced by its limiting value if $\alpha$ is an integer or zero.\\
 By using the familiar integral representation 
 \begin{equation}
 \displaystyle K_\alpha (x)=\int_0^{+\infty} e^{-x \cosh(t)} \cosh(\alpha t)dt ,\quad \alpha \in \mathbb{R},\quad x>0,
  \end{equation}
we deduce that   
 \begin{equation}
 \displaystyle K_\alpha (x)>0,\quad \forall \alpha\in \mathbb{R}  ,\quad \forall x>0.
 \end{equation}
 
Bell \cite{Be} showed that
\begin{equation}
\displaystyle K_\alpha (x)=\frac{1}{\sqrt{\pi}\Gamma(\alpha+\frac{1}{2})}\left(\frac{x}{2}\right)^\alpha
\int_1^{+\infty} e^{-x t} \left(t^2-1\right)^{\alpha-\frac{1}{2}}dt ,
\quad \alpha >- \frac{1}{2},\quad x>0.
\end{equation}
 
  We  denote by $C$ the set of continuous functions and $C_0$ its subspace of functions vanishing at infinity, 
   $S$ the Schwartz space of infinitely differentiable and rapidly decreasing functions, $\mathcal{D}$ the space of infinitely differentiable even functions with compact support,
  $L^p$ the set of $p$-power integrable  functions with respect to the measure $ dx$ on $\mathbb{R}$,
   $L_\alpha^p$ the set of $p$-power integrable even functions with respect to the measure $x^{2\alpha +1} dx$ on $[0,\infty[$. The symbol $M^+$ stands for the set of nonnegative finite Borel measure on $\mathbb{R}$.
  
 \section{Harmonic analysis related to Bessel translation operator} 
 In this section we present some properties of the Bessel transform, the Bessel translation operator and the Bessel convolution product \cite{Fi}.\\
  
 \noindent The Fourier-Bessel transform $\mathcal{F}_\alpha$ is defined for $f\in D(\mathbb{R})$ by: 
  \begin{equation}
  \displaystyle \mathcal{F}_\alpha f(x)=c_\alpha \int_0^{+\infty} f(t) j_\alpha (xt) t^{2\alpha+1} dt, 
  \quad x\geq 0,
  \end{equation}
   where 
  \begin{equation}
   \displaystyle c_\alpha =\frac{1}{2^\alpha\Gamma (\alpha+1)}
  \end{equation}\\
 \begin{theorem}\label {t1}
 1) For $f\in L_\alpha^2$, we have $\mathcal{F}_\alpha (f)\in L_\alpha^2$ and 
 \begin{equation}
 \displaystyle   \left\|\mathcal{F}_\alpha (f)\right\|_2=\left\|f\right\|_2.
 \end{equation}
 2) For $f\in L_\alpha^1$, we have $\mathcal{F}_\alpha (f)\in C_0$ and 
 \begin{equation}
 \displaystyle  \left\|\mathcal{F}_\alpha (f)\right\|_\infty\leq\left\|f\right\|_1.
 \end{equation}
 3) Let $1<p\leq 2$ and $p'=\frac{p}{p-1}$. For $f\in L_\alpha^p$, 
 we have $\mathcal{F}_\alpha (f) \in L_\alpha^{p'}$
 and 
 \begin{equation}
 \displaystyle \left\|\mathcal{F}_\alpha (f)\right\|_{p'}\leq 
 \left[\frac{p^{\frac{1}{p}}}{p'^{\frac{1}{p'}}}\right] \left\|(f)\right\|_{p}.
  \end{equation}
  \end{theorem}
  
 \noindent The Bessel translation $T_x^\alpha$  is defined for  $f\in L_\alpha^p, p\geq 1$, a.e by:
  \begin{equation}
 \displaystyle T_x^\alpha f(y)=c_\alpha \int_0^{+\infty} f(z)D(x,y,t) t^{2\alpha+1} dt,\quad x\neq 0.
  \end{equation}
  $$\displaystyle T_0^\alpha f(y)=f(y),$$
 
 where 
 \begin{equation}
 \displaystyle  D(x,y,t)=\frac{2^{3\alpha-1}\Gamma^2 (\alpha+1)}{\sqrt{\pi}\Gamma (\alpha+\frac{1}{2})}
 \frac{\left[\Delta(x,y,t)\right]^{2\alpha-1}}{\left(xyt\right)^{2\alpha}},
 \end{equation}
 $\Delta(x,y,t)$ is the area of the  triangle  $[x,y,t]$ if this  triangle exists and 0 if not. We have 
 $$ \displaystyle c_\alpha \int_0^{+\infty} D(x,y,t) t^{2\alpha+1} dt=1.$$
 
  \noindent If $f$ is continuous on $[0,+\infty[$, we have 
 \begin{equation}
 \displaystyle T_x^\alpha f(y)=\frac{\Gamma(\alpha +1)}{\sqrt{\pi}\Gamma(\alpha +\frac{1}{2})}
 \int_0^{\pi} f\left(\sqrt{x^2+y^2+2xy\cos \theta}\right)(\sin \theta )^{2\alpha} d\theta .
 \end{equation}

 \begin{theorem}
 
\noindent 1) For $\lambda , x, y \in [0,+\infty[$ and $f\in L_\alpha^1$, we have
\begin{equation}
\displaystyle T_x^\alpha j_\alpha (\lambda y) =j_\alpha (\lambda x)j_\alpha (\lambda y),
\end{equation}
\begin{equation}
\displaystyle \mathcal{F}_\alpha \left(T_x^\alpha (f)\right)(y)=j_\alpha (xy)\mathcal{F}_\alpha (f)(y).
\end{equation}

\noindent 2) For $\displaystyle f\in L_\alpha^p$ , we have $\displaystyle T_x^\alpha f(y) \in L_\alpha^p$ and
 \begin{equation}
 \displaystyle \left\|T_x^\alpha f\right\|\leq \left\|f\right\|_p .
 \end{equation}
 \end{theorem} 
  
\noindent The Bessel convolution product is defined for $f,g \in L_\alpha^1$  by: 
 \begin{equation}
\displaystyle f*_\alpha g (x)=c_\alpha\int_0^{+\infty}T_x^\alpha f(y) g(y) y^{2\alpha+1} dy.
\end{equation}
\begin{propo}
Let $f$ and $g$ be in $L_\alpha^1$, then 
 \begin{equation}
\displaystyle f*_\alpha g \in L_\alpha^1
\end{equation}
and 
 \begin{equation}
\displaystyle \mathcal{F}_\alpha \left(f*_\alpha g \right)=\mathcal{F}_\alpha (f)\mathcal{F}_\alpha (g).
\end{equation}
\end{propo}

 \begin{theorem}
 1) Let $1\leq r\leq +\infty$, $f\in L_\alpha^r$ and $g\in L_\alpha^1$, then 
 \begin{equation}
 \displaystyle \left\|f*_\alpha g \right\|_r\leq \left\|f\right\|_r\left\|g\right\|_1,\quad \qquad\qquad\qquad
 \end{equation}
 \begin{equation}
 \displaystyle \left\|f*_\alpha g \right\|_\infty \leq \left\|f\right\|_r\left\|g\right\|_{r'}, \quad \frac{1}{r}+\frac{1}{r'}=1.
 \end{equation}
\noindent 2) Let $p, q, r$ be in $]1,2]$, such that $\displaystyle \frac{1}{r}=\frac{1}{p}+\frac{1}{q}-1$.\\
For $f\in L_\alpha^p$ and $g\in L_\alpha^q $ , we have 
\begin{equation}
f*_\alpha g \in L_\alpha^r
\end{equation}
and
\begin{equation}
\displaystyle \left\|f*_\alpha g \right\|_r\leq B_p B_q B_{r'} \left\|f\right\|_p\left\|g\right\|_{g}, \quad \frac{1}{r}+\frac{1}{r'}=1,
 \end{equation}
 where $B_m=\left[\frac{m^{\frac{1}{m}}}{m'^{\frac{1}{m'}}}\right]^{\alpha +1}$.

\end{theorem}

\section{ strictly positive definite functions related to Bessel functions}

A complex valued continuous function $f$ is said positive definite (resp. strictly positive definite ) 
 on $\mathbb{R}$ if for every  reals numbers $x_1, x_2, ...x_N$ and every complex numbers 
 $z_1, z_2, ...z_N$ not all zero, the inequality 
 \begin{equation}
  \displaystyle \sum_{j=1}^{N}\sum_{k=1}^{N} z_j \bar{z_k} f(x_j-x_k) \geq 0 \quad {\rm(resp. >0)}
  \end{equation}
 
\noindent holds true.\\
In this section we give some conditions for a class of functions related to Bessel functions 
 to be positive definite or  strictly positive definite.\\
 
\begin{propo}
For $\alpha >-\frac{1}{2}$, we have\\
\noindent 1) 
\begin{equation}
j_\alpha \in \mathcal{P}.
\end{equation}
\noindent 2) Let $\mu \in M^+$ such that $\hat{\mu} \in \mathcal{P}^s$. Then the function 
\begin{equation}
\int_{\mathbb{R}} j_\alpha (x\xi) d\mu (x) \in \mathcal{P}^s.
\end{equation}

\end{propo}

\begin{proof}
1) It's known that \cite{Wa}
\begin{equation}
\displaystyle j_\alpha (x)=\frac{1}{2^{\alpha-1} \sqrt{\pi}\Gamma (\alpha+\frac{1}{2})} 
  \int_0^1 \left(1-t^2\right)^{\alpha-\frac{1}{2}} \cos (xt) dt.
  \end{equation}
  By Bochner's theorem we conclude.\\
  
\noindent 2) Since $j_\alpha \in \mathcal{P}$, hence the proof is done by corollary 6.6 in \cite{D}. 
\end{proof}

\begin{theorem}\label{K}
 For $\alpha>0$, the even function $x^\alpha K_\alpha (x)$ is strictly positive definite on $\mathbb{R}$.
 \end{theorem}
 \begin{proof} For $\alpha>0$, the even function $x^\alpha K_\alpha (x)$ admit the Basset's integral representation 
 (\cite{Wa}, p 172):
 \begin{equation}
 \displaystyle x^\alpha K_\alpha (x)=\frac{2^\alpha \Gamma(\alpha+\frac{1}{2})}{\Gamma(\frac{1}{2})}
  \int_0^\infty \frac{\cos (xt)}{(1+t^2)^{\alpha +\frac{1}{2}}} dt, \quad x\in\mathbb{R}.
  \end{equation}
  By Bochner's theorem we have 
  \begin{equation}
  \displaystyle x^\alpha K_\alpha (x)\in \mathcal{P}.
  \end{equation}
Using the Bell integral representation \cite{Be}, we obtain  
 \begin{equation}
 \displaystyle x^\alpha K_\alpha (x)\in L_1.
 \end{equation}
 
 we conclude by theorem 6.5 in \cite{D}
 \end{proof}
\begin{theorem}
1) Let $\varphi$ be a nonnegative function in  $ L_\alpha ^1$  , then
\begin{equation}
\displaystyle \mathcal{F}_\alpha (\varphi) \in  \mathcal{P}.
\end{equation}

\noindent  2)  Let $\varphi$ be a nonnegative continuous function in  $ L_\alpha ^1$ and not identically zero, then
\begin{equation}
\displaystyle \mathcal{F}_\alpha (\varphi) \in  \mathcal{P}^s.
\end{equation}

\end{theorem}
\begin{proof}
1) Let $x_1, x_2,...,x_n$ are reals numbers and $z_1, z_2,...,z_n$ are complexes numbers , we have 
$$\displaystyle  \sum_{j=1}^{n} \sum_{k=1}^{n} z_j \overline{z_k}\mathcal{F}_\alpha (\varphi) (x_j-x_k) =
\int_0^{+\infty} \sum_{j=1}^{n} \sum_{k=1}^{n} z_j\sqrt{\varphi (t)} \overline{(z_k \sqrt{\varphi (t)})} 
 j_\alpha (x_jt-x_kt) d\mu_\alpha (t).$$
Since $j_\alpha \in \mathcal{P}$, we  conclude.\\

\noindent 2) By 1), we have $\displaystyle \mathcal{F}_\alpha (\varphi) \in  \mathcal{P}$. Bochner's theorem 
asserts that there exists a nonnegative  finite Borel measure $\mu$ such that 
$\displaystyle \mathcal{F}_\alpha (\varphi)=\hat{\mu}$, with $\hat{\mu}(0)=\mathcal{F}_\alpha (\varphi)(0)>0$.
 By the Riemann-Lebesgue lemma, we have $\displaystyle \lim_{\left|\xi\right|\leftarrow \infty}\mathcal{F}_\alpha (\varphi)(\xi)=0$. Then the support of $\mu$ must contain an interior point and hence by theorem 6.8 in \cite{We}, 
  we conclude . 
 
\end{proof}

\noindent {\bf Example 1:} For $0<\alpha <\beta$ and $a>0$, the even function
\begin{equation} 
\displaystyle \phi_{\alpha,\beta}(x)=\frac{a^{\alpha-\beta}}{2^\alpha \Gamma(\alpha+1)} x^{\beta-\alpha} K_{\beta-\alpha} (ax)
 \in \mathcal{P}^s,
 \end{equation}
  where $K_\alpha$ is the modified Bessel function of second kind.\\

\noindent {\bf Proof :} For $0<\alpha<\beta$ and $a>0$, we put
$$\displaystyle \varphi_\beta (x)=\frac{1}{\left(x^2+a^2\right)^{\beta +1}}.$$
Then, $\displaystyle \varphi_\beta \in L_\alpha ^1$ is continuous positive function. By (\cite{An}, p 254), we have:
$$\displaystyle \mathcal{F}_\alpha (\varphi_\beta)(x)=\frac{1}{2^\alpha \Gamma(\alpha+1)}
\int_0^\infty \frac{1}{(t^2+a^2)^{\beta +1}}  j_\alpha(xt)t^{2\alpha+1} dt =\frac{a^{\alpha-\beta}}{2^\alpha \Gamma(\alpha+1)} x^{\beta-\alpha} K_{\beta-\alpha} (ax). $$
Thus, the proof is done by the last theorem.\\

\noindent {\bf Example 2:} For $-1<\alpha<\beta$ and $a\neq 0$, the function 
\begin{equation}
\displaystyle \phi_{\alpha,\beta}(x)=a^{\beta+2} {_1F_1}(1+\frac{\beta}{2};\alpha +1; -\frac{1}{4}a^2x^2)
\in \mathcal{P}^s,
\end{equation}
 where $_1F_1$ is the hypergeometric function.\\
\noindent {\bf Proof :}
For $-1<\alpha<\beta$ and $a\neq 0$, we consider the even function 
$$\displaystyle \varphi_{\alpha,\beta} (x)=x^{\beta-\alpha}  e^{-\frac{x^2}{a^2}}, \quad x\geq  0.$$
 Then, $\varphi_{\alpha,\beta} \in  L_\alpha ^1 $ is continuous, nonnegative function 
  and
$$\displaystyle \phi_{\alpha,\beta}(x)= \mathcal{F}_\alpha (\varphi_{\alpha,\beta})(x),$$
hence the proof is done by the last theorem.\\
\begin{coro}
 For $\alpha \geq -\frac{1}{2}$,  let $\varphi \in L_\alpha ^1$  be a nonnegative function, then 
\begin{equation}
\mathcal{F}_\alpha (T_x\varphi)\in \mathcal{P}.
\end{equation}

\end{coro}
\begin{proof} We have:
$$\displaystyle  \mathcal{F}_\alpha (T_x\varphi)=j_\alpha (x.)\mathcal{F}_\alpha (\varphi).$$
Since $j_\alpha \in \mathcal{P}$, the last theorem complete the proof.
\end{proof}

\noindent We consider the Wiener algebra
\begin{equation}
\displaystyle \mathcal{A}_\alpha =\left\{f\in L_\alpha^1 / \mathcal{F}_\alpha(f) \in  L_\alpha^1\right\}.
\end{equation}

\begin{theorem}
Let $\displaystyle \varphi \in \mathcal{A}_\alpha \cap C$, 
 $\displaystyle \mathcal{F}_\alpha (\varphi )\geq 0$ and $\varphi$ non vanishing.
 Then $\displaystyle \varphi \in \mathcal{P}^s$.

\end{theorem}

\begin{proof}
Let $\displaystyle \varphi \in \mathcal{A}_\alpha \cap C$, by the inversion formula we have 
$$\displaystyle \varphi =\mathcal{F}_\alpha \left(\mathcal{F}_\alpha (\varphi)\right).$$
Since   $\displaystyle \mathcal{F}_\alpha (\varphi )\geq 0$, then $\varphi \in \mathcal{P}$. Moreover 
$\displaystyle \varphi \in L_\alpha^1$ and $\varphi $ non vanishing , then $\varphi \in \mathcal{P}^s$.\\

\end{proof}

\begin{coro}
Let $\displaystyle \varphi \in \mathcal{A}_\alpha \cap C$, 
 $\displaystyle \mathcal{F}_\alpha (\varphi )\geq 0$ and $\varphi$ non vanishing.
 Then 
 \begin{equation}
 \displaystyle T^t\varphi =E_t*_\alpha \varphi\in \mathcal{P}^s
\end{equation}
and 
 \begin{equation}
 \displaystyle P^t\varphi =p_t*_\alpha \varphi\in \mathcal{P}^s,
\end{equation}
where $E_t$ and $p_t$ are respectively the Gauss and the Poisson kernels associated to the Bessel operator given by:
\begin{equation}
\displaystyle E_t(x)=\frac{e^{-\frac{x^2}{4t}}}{(2t)^{\alpha+1}},
\end{equation}

\begin{equation}
\displaystyle p_t(x)=\frac{2^{\alpha+1}\Gamma(\alpha+\frac{3}{2})}{\sqrt{\pi}}\frac{t}{(t^2+x^2)^{\Gamma(\alpha+\frac{3}{2})}}.
\end{equation}
\end{coro}

\section{Relation between strictly positive definite function and logarithmically completely monotonic functions}
\subsection{Definitions and properties:}
\begin{defin} A function $f$ is said to be completely monotonic (CM) on an interval $I\subset\mathbb{R}$
 if $f\in C(I)$ has derivatives of all orders on $I^0$ and, for all $n\in \mathbb{N}$,\\
\begin{equation}
 \displaystyle (-1)^n f^{(n)}(x)\geq 0, \quad x\in I^0;\quad  n\in \mathbb{N}.
 \end{equation}
 
\noindent The class of all  completely monotonic functions on  $I$ is denoted by $CM(I)$.
\end{defin}

\begin{lemme}
The sum and the product of completely monotonic functions are also.
\end{lemme}
\begin{defin}
A function $f$ is said to be logarithmically completely monotonic (LCM) on an interval $I\subset\mathbb{R}$
 if 
 $$f>0, \quad f\in C(I),$$ 
 has derivatives of all orders on $I^0$ and, for all $n\in \mathbb{N}\setminus \{0\}$,\\
 
 \begin{equation}
 \displaystyle (-1)^n [{\rm ln} f(x)]^{(n)}\geq 0, \quad x\in I^0;\quad  n\geq 1.
 \end{equation}
 
\noindent The class of all logarithmically completely monotonic on  $I$ is denoted by $LCM(I)$.

\end{defin}
\begin{propo}
\begin{equation}
LCM(I)\subset CM(I) .
\end{equation}

\end{propo}

\begin{defin}
A function $\varphi: [0, \infty )\rightarrow \mathbb{R}$ is said to be positive definite on  $\mathbb{R}^d$ 
if the corresponding multivariate function  $\phi := \varphi (\left\|.\right\|_2^2)$ is positive definite on every $\mathbb{R}^d$.
\end{defin}

\noindent In \cite{Sc} Schoenberg establish the connexion between positive definite radial and completely 
monotone functions.\\

 \begin{theorem} {\rm({\bf Schoenberg})}
 A function $\varphi$ is completely monotone on $[0, \infty )$ if and only if 
 $\phi := \varphi (\left\|.\right\|_2^2)$ is positive definite on every $\mathbb{R}^d$.
 \end{theorem}
 
\noindent In \cite{We} Wendland  was interested by strictly positive definite functions and present a complete
 characterization of radial functions as being strictly positive definite on every $\mathbb{R}^d$.
\begin{theorem}\label{W} {\rm({\bf Wendland })}
For a function $\varphi: [0, \infty )\rightarrow \mathbb{R}$ the following three properties are equivalent\\
\noindent (1)  $\varphi$ is strictly positive definite on every $\mathbb{R}^d$;\\
(2)  $\varphi (\sqrt{.})$ is completely monotone on $[0, \infty )$ and non constant;\\
(3) there exists a finite nonnegative Borel measure $\nu$ on $[0, \infty )$ that is non concentrated at zero, 
such  that
$$\varphi (r)=\int_0^\infty e^{-r^2 t} d\nu(t).$$

\end{theorem}

\begin{propo} For $\alpha>0$, the even function
\begin{equation}\label{cmk}
\displaystyle x^{\frac{\alpha}{2}} K_\alpha (\sqrt{x})\in CM(]0,\infty[),
\end{equation}
 where $K_\alpha$ is the modified Bessel function of second kind.
 \end{propo}
\begin{proof}
By Theorem \ref {K}, if $\alpha>0$, then
$$\displaystyle x^\alpha K_\alpha (x)\in \mathcal{P}^s .$$
Since  this  function is $C^\infty (\mathbb{R})$ and positive on $]0,\infty[$, we conclude.
\end{proof}

\noindent The next result exist in \cite{Sr}. We give an elementary proof.
 \begin{theorem} Suppose that 
 $$\displaystyle f\in C(I), \quad f>0 \quad {\rm and} \quad f'\in {\rm CM(I)}.$$
 Then 
 $$\displaystyle \frac{1}{f}\in {\rm LCM(I)}.$$
 
 \end{theorem}
 \begin{proof} Proving by induction that:
 $$\displaystyle (-1)^n \left({\rm ln}(\frac{1}{f})\right)^{(n)}\geq 0,\quad \forall n\geq 1.$$
 For $n=1$, we have 
 $$\displaystyle (-1)\left({\rm ln}(\frac{1}{f})\right)^{(1)}=\frac{f'}{f}\geq 0.$$
 Suppose that 
 $$\displaystyle (-1)^k \left({\rm ln}(\frac{1}{f})\right)^{(k)}\geq 0,\quad \forall 1\leq k \leq n.$$ 
 Put $\displaystyle f'=f \frac{f'}{f}$. By the Leibnitz formula, we get 
 
 $$\displaystyle \left(f'\right)^{(n)} =-\left[f\left({\rm ln}(\frac{1}{f})\right)^{(1)}\right]^{(n)} 
  =\sum_{k=0}^n C_n^k f^{(k)} \left({\rm ln}(\frac{1}{f})\right)^{(n-k+1)}.$$
 Then 
 $$\displaystyle (-1)f \left({\rm ln}(\frac{1}{f})\right)^{(n+1)}=\left(f'\right)^{(n)} +
 \sum_{k=1}^n C_n^k f^{(k)} \left({\rm ln}(\frac{1}{f})\right)^{(n-k+1)},$$
which readily yields 
  $$\displaystyle (-1)^{n+1}f \left({\rm ln}(\frac{1}{f})\right)^{(n+1)}=(-1)^{n}\left(f'\right)^{(n)} +\qquad \qquad\qquad\qquad\qquad\qquad\qquad\qquad\qquad\qquad\qquad$$
 $$\qquad\qquad \sum_{k=1}^n C_n^k \left[(-1)^{k-1} (f')^{(k-1)}\right]
 \left[(-1)^{n-k+1} \left({\rm ln}(\frac{1}{f})\right)^{(n-k+1)}\right].$$
 Since $f>0$, $f'\in {\rm CM(I)}$ and $\displaystyle (-1)^k \left({\rm ln}(\frac{1}{f})\right)^{(k)}\geq 0,\quad \forall 1\leq k \leq n,$ we obtain 
 
 $$\displaystyle (-1)^{n+1}f \left({\rm ln}(\frac{1}{f})\right)^{(n+1)}\geq 0,$$
 which complete the proof.
\end{proof}
\begin{theorem}
Let $\varphi : [0,\infty[\rightarrow \mathbb{R}$  \quad  and  \quad $\varphi'$ non vanishing .\\
Suppose that 
 $$\displaystyle \varphi=\hat{\mu}\in \mathcal{P}^s\cap C^\infty.\quad $$
 Then
 $$\displaystyle -\frac{1}{\varphi'(\sqrt{.})}\in {\rm LCM}(0,\infty).$$

\end{theorem}
\begin{proof}
Since $\varphi\in \mathcal{P}$, by Bochner's theorem  there exist $\mu\in M^+$ such that $\varphi =\hat{\mu}$.\\

\noindent Since $\varphi \in C^\infty$, By corollary 6.3 in \cite{D}, we have 
 $$\displaystyle x^2 \mu \in M^+ {\rm and } -\varphi'' \in \mathcal{P}^s$$
 Let 
 $$\displaystyle f=-\varphi'(\sqrt{.}),$$
then 
$$\displaystyle f'(x)=-\frac{1}{2\sqrt{x}} \varphi''(\sqrt{x}).$$
Since $-\varphi'' \in \mathcal{P}^s$, then $-\varphi''(\sqrt{.})\in {\rm CM}(0,\infty)$.\\
 Moreover $\frac{1}{\sqrt{.}} \in {\rm CM}(0,\infty)$ , hence $f'\in {\rm CM}(0,\infty)$.
Since $f>0$, we conclude by the last theorem.
 
 \end{proof}
 \begin{coro}
 Let $\varphi : [0,\infty[\rightarrow \mathbb{R}$  \quad  and  \quad $\varphi^{(n)}$ non vanishing for all 
 $n\geq 1$.\\
Suppose that 
 $$\displaystyle \varphi=\hat{\mu}\in \mathcal{P}^s\cap C^\infty.\quad $$
 Then
 $$\displaystyle -\frac{1}{\varphi^{(n)}}\in \mathcal{P}^s\quad \forall n\geq 1.$$
\end{coro}
 \begin{proof} By induction :\\
 Since $\varphi=\hat{\mu}\in \mathcal{P}^s\cap C^\infty$ and $\varphi'$ non vanishing, 
 then $$-\frac{1}{\varphi'(\sqrt{.})}\in {\rm CM}.$$
  Thus by Wendland theorem
  $$-\frac{1}{\varphi'}\in \mathcal{P}^s.$$
  Suppose that 
 $$\displaystyle \phi=-\frac{1}{\varphi^{(n)}}\in \mathcal{P}^s,\quad \forall n\geq 1.$$ 
 Then 
$$\displaystyle -\frac{1}{\phi'}\in \mathcal{P}^s.$$ 
 Hence
$$\displaystyle  -\frac{1}{\varphi^{(n+1)}}= -\frac{1}{\phi'} \left[-\frac{1}{\varphi^{(n)}}\right]^2 \in \mathcal{P}^s,$$
 witch complete the proof.
 \end{proof}
 \section{Applications}
 \subsection{A class of logarithmically completely monotonic functions related to the modified Bessel 
 functions of second kind}
 \begin{theorem}
 For $\alpha >0$, the function
 \begin{equation}
 \displaystyle \frac{1}{x^{\frac{\alpha+1}{2}}K_\alpha (\sqrt{x})} \in LCM(]0,\infty[).
 \end{equation}
 
 \end{theorem}
 \begin{proof} 
 By theorem 6.4 in \cite{Be}, we have 
  \begin{equation}
  \displaystyle  \frac{d}{dx} \left[x^{\alpha+1} K_{\alpha+1} (x)\right]=-x^{\alpha+1} K_\alpha (x).
 \end{equation}
 By proposition \ref{K}, we have  
  $$\displaystyle x^{\alpha+1} K_{\alpha+1} (x)\in \mathcal{P}^s \cap C^\infty$$
  and is non vanishing. Theorem 10 complete the proof.
   \end{proof}
   
 \begin{propo}
 For $\alpha >0$ and $\beta >0$,
 \begin{equation}
 \displaystyle g_{\alpha,\beta}(x)=\frac{K_{\alpha+1}(\sqrt{x})}{x^\beta K_\alpha (\sqrt{x})}\in CM(]0,\infty[).
 \end{equation}
 \end{propo} 
 \begin{proof}
  For $\alpha >0$ and $\beta >0$, using proposition 4 and the last proposition we have 
   $$\displaystyle \frac{1}{x^\beta} \in CM(]0,\infty[),$$
   $$\displaystyle x^{\frac{\alpha+1}{2}}  K_{\alpha +1} (\sqrt{x})\in CM(]0,\infty[),$$
   and 
   $$\displaystyle \frac{1}{x^{\frac{\alpha+1}{2}}K_\alpha (\sqrt{x})} \in LCM(]0,\infty[)\subset CM(]0,\infty[).$$
   The product gave the result.
 \end{proof}
  \begin{propo}
    For $\alpha >0$,
  \begin{equation}
  \displaystyle K_\alpha (\sqrt{x}) \in LCM(]0,\infty[).
  \end{equation}
 \end{propo} 
 
 \begin{proof}
 In \cite{Is}, Ismail gave the integral representation 
 \begin{equation}
 \displaystyle \frac{K_{\alpha-1} (\sqrt{x})}{\sqrt{x}K_\alpha (\sqrt{x})}=\frac{4}{\pi^2} 
 \int_0^\infty  \frac{t^{-1} dt}{(x+t^2) \left[J_\alpha^2(t)+Y_\alpha^2(t)\right]}, \quad x>0, \quad \alpha \geq 0.
  \end{equation}
 Using the relationship , Watson (\cite{Wa}, p: 79)
 \begin{equation}
 \displaystyle K_\alpha '(x)=-\frac{1}{2} \left\{K_{\alpha-1}(x)+K_{\alpha+1}(x)\right\} \quad x>0,\quad \alpha \geq 0,
  \end{equation}
 we find 
\begin{equation}
\displaystyle \left[{\rm Ln} \left(K_\alpha (\sqrt{x}) \right)\right]'=
 -\frac{1}{2}\frac{K_\alpha '(\sqrt{x})}{\sqrt{x}K_\alpha (\sqrt{x})}=
    -\frac{1}{4} \left\{\frac{K_{\alpha-1} (\sqrt{x})}{\sqrt{x}K_\alpha (\sqrt{x})}+\frac{K_{\alpha+1} (\sqrt{x})}{\sqrt{x}K_\alpha (\sqrt{x})}\right\}.
\end{equation}
By the last integral representation,  It's clear that   
 $$\displaystyle \frac{K_{\alpha-1} (\sqrt{x})}{\sqrt{x}K_\alpha (\sqrt{x})}\in CM(]0,\infty[).$$
 Since 
 $$\displaystyle \frac{K_{\alpha+1} (\sqrt{x})}{\sqrt{x}K_\alpha (\sqrt{x})}=
 g_{\alpha,\frac{1}{2}}(x)\in CM(]0,\infty[).$$
 Thus 
 $$\displaystyle -\left[{\rm Ln} \left(K_\alpha (\sqrt{x}) \right)\right]'\in CM(]0,\infty[).$$ 
 and
 $$\displaystyle K_\alpha (\sqrt{x}) \in LCM(]0,\infty[).$$
 \end{proof}
  \begin{propo}
    Let $\alpha >0$, $x_0>0$ and $K_\alpha (\sqrt{x_0})=1$, then 
  \begin{equation}
  \displaystyle \Delta_\alpha(x)=\frac{-1}{{\rm Ln}\left(K_\alpha (\sqrt{x})\right)} \in LCM(]x_0,\infty[).
  \end{equation}
 \end{propo} 
 \begin{proof} Put 
 $$\displaystyle y_\alpha(x)= -{\rm Ln}\left(K_\alpha (\sqrt{x})\right), \quad x>x_0.$$
 Since $K_\alpha$ is decreasing on $]x_0, +\infty[$, we have 
 $$\displaystyle 0<K_\alpha (\sqrt{x})<K_\alpha (\sqrt{x_0})=1,\quad x>x_0.$$
 We have prove in the last proposition that 
 $$\displaystyle y_\alpha'(x)=-\left[{\rm Ln} \left(K_\alpha (\sqrt{x}) \right)\right]'\in CM(]x_0,\infty[).$$ 
 Theorem 9 completes the proof.
\end{proof} 
\subsection{Inequalities for the modified Bessel function of the second kind}
\begin{defin}{\bf (logarithmically-convex)} A positive function defined on a interval $I$ is said to be 
log-convex if ${\rm Ln} (f)$ is convex, i.e for all $x,y \in I$ and $\lambda \in [0,1]$, we have
 \begin{equation}
 \displaystyle f(\lambda x+(1-\lambda)y)\leq \left[f(x)\right]^\lambda \left[f(y)\right]^{1-\lambda}.
 \end{equation}
\end{defin}

\noindent We have immediately the following result:

\begin{lemme}
If $f$ is LCM function on an interval $I\subset \mathbb{R}$, then $f$ is log-convex on $I$.
\end{lemme}
\begin{theorem}
 For $\alpha>0$,\\
 
\noindent 1) if $x,y>0$, then
 \begin{equation}
 \displaystyle K_\alpha\left(\sqrt{\frac{x+y}{2}}\right)\leq \sqrt{K_\alpha(\sqrt{x})K_\alpha(\sqrt{y})}
 \leq \left(\frac{x+y}{2\sqrt{xy}}\right)^{\frac{\alpha+1}{2}} K_\alpha\left(\sqrt{\frac{x+y}{2}}\right).
 \end{equation}
 
\noindent 2)  if $x,y>x_0>0$, where  $K_\alpha(\sqrt{x_0})=1$, then
 \begin{equation}
   \displaystyle  \sqrt{-{\rm Ln} \left(K_\alpha (\sqrt{x}) \right)}\sqrt{-{\rm Ln} \left(K_\alpha (\sqrt{y}) \right)}\leq -{\rm Ln}\left( K_\alpha \left(\sqrt{\frac{x+y}{2}}\right)\right).
\end{equation}
 In each of the above inequalities equality hold if and only if $x=y$.

\end{theorem}
\begin{proof}
1) By Proposition 6, we know that for $\alpha>0$, the function 
$$\displaystyle K_\alpha (\sqrt{x}) \in LCM(]0,\infty[).$$ 
Thus it is log-convex on $]0,\infty[$. We get
$$\displaystyle K_\alpha\left(\sqrt{\frac{x+y}{2}}\right)\leq \sqrt{K_\alpha(\sqrt{x})K_\alpha(\sqrt{y})}.$$ 
By Theorem 11, we know that for $\alpha>0$, the function  
$$\displaystyle \frac{1}{x^{\frac{\alpha+1}{2}}K_\alpha (\sqrt{x})} \in LCM(]0,\infty[).$$
 Thus, it is log-convex on $]0,\infty[$. We get
 $$\displaystyle \sqrt{K_\alpha(\sqrt{x})K_\alpha(\sqrt{y})}\leq \left(\frac{x+y}{2\sqrt{xy}}\right)^{\frac{\alpha+1}{2}} K_\alpha\left(\sqrt{\frac{x+y}{2}}\right) .$$
2) By Proposition 7, we have, for $x,y>x_0>0$, where  $K_\alpha(\sqrt{x_0})=1$ 
$$\displaystyle \Delta_\alpha(x)=\frac{-1}{{\rm Ln}\left(K_\alpha (\sqrt{x})\right)} \in LCM(]x_0,\infty[).$$
 Thus, it is log-convex on $]x_0,\infty[$. We get: 
  $$\displaystyle \Delta_\alpha (\frac{x+y}{2})\leq \Delta_\alpha(x)\Delta_\alpha(y),$$
  which completes the proof.
\end{proof}


\begin{thebibliography}{99}
\bibitem{An} G. Andrews, R. Askey and Q. Ranjan,{\it  special function}, Cambridge University Press, (1999).


\bibitem{Bek} W. Bekner, {\it Inequalities in Fourier analysis}, Ann. Math, (2), 102 (1975), 159-182.

\bibitem{Be}W.W. Bell , {\it Special functions for scientists and engineers}, London 1967.
Encyclopedia of Mathematics and its application, Vol 35 Cambridge
Univ. Press, Cambridge, UK, 1990.

\bibitem{Bo} S. Bochner, {\it Integral transforms and their applications}.
Applied Math. Sciences 25. Springer-Verlag. New York Berlin
Heidelberg Tokyo.

\bibitem{D}  F. Derrien,  {\it Strictly positive definite functionsd on the real line}, hal-00519325,version 1-20sep 2010.  

\bibitem{Fa} M. Ky Fan, {\it Les fonctions d\'efinies positives et les fonctions compl\`etement monotones }, 
 Memorial Sciences Math\'ematiques Paris, (1950).

\bibitem{Fi} A. Fitouhi, {\it In\'egalit\'e de Babenko et in\'egalit\'e logarithmique de Sobolev pour l'op\'erateur de Bessel},C.R.Acad. Sci. Paris, 305(I) (1987) 877–880.

\bibitem{Sr} S. Guo and H.M. Srivastava , {\it A certain function class related to the class 
of logarithmically completely monotonic functions }, Mathematical and Computer Modelling 49 
(2009), 2073-2079.

\bibitem{Is} M.E.H. Ismail, {\it Bessel functions and the infinite divisibility of the student t-distribution}, Ann. Prob. 5 (1977), 582-585. 


\bibitem{Ism} M.E.H. Ismail, {\it Complete monotonicity of the modified Bessel functions}, Proceeding of the American  Mathematical society, Vol. 108, No 2(1990), 353-361.


\bibitem{Sc} I.J. Schoenberg, Metric spaces and completely monotonic functions. Ann. Math. 39, (1938), 811-841. 

\bibitem{tit} E.C. Titchmarsh. {\it Introduction to The Theory of Fourier
Integrals}, Oxford University Press, second ed. 1937.

\bibitem{Wa} G.N. Watson, {\it A treatise on the theory of Bessel functions }, Cambridge University Press.

\bibitem{We} H. Wendland , {\it Scattered data approximations },
 Cambridge University Press. , Cambridge, 2005.

\bibitem{Wi} D.V. Widder, {\it The Laplace transform}, Princeton Univ Press, Princeton, NJ, 1941.
\end{thebibliography}
\end {document}